\documentclass[12pt,reqno]{amsart}

\usepackage{amssymb, xcolor}
\usepackage{amscd}
\usepackage{amsfonts}
\usepackage{setspace}
\usepackage{version}

\usepackage{graphicx}

\newtheorem{theorem}{Theorem}[section]
\newtheorem{lemma}[theorem]{Lemma}
\newtheorem{proposition}[theorem]{Proposition}
\newtheorem{corollary}{Corollary}
\newtheorem{rem}{Remark}

\theoremstyle{definition}

\renewcommand{\leq}{\leqslant}
\renewcommand{\geq}{\geqslant}

\def\R{\mathbb{R}}

\def\Z{\mathbb{Z}}

\def\N{\mathbb{N}}

\def\overast{\overline{\ast}}

\numberwithin{equation}{section}

\begin{document}

\title[Sumset inequalities on cubes]{
Discrete Brunn--Minkowski  inequality for subsets of the cube
}




\author[Becker]{Lars Becker}
\author[Ivanisvili]{Paata Ivanisvili}
\author[Krachun]{Dmitry Krachun}
\author[Madrid]{Jos\'e Madrid}

\address[LB]{Mathematical Institute, University of Bonn, Endenicher Allee 60, 53115, Bonn, Germany
}
\email{becker@math.uni-bonn.de}

\address[PI]{Department of Mathematics, University of California, Irvine,  CA 92617, USA
}
\email{pivanisv@uci.edu}

\address[DK]{Mathematics Department, Princeton University, Princeton, NJ 08544, USA
}
\email{dk9781@princeton.edu}

\address[JM]{Department of  Mathematics, Virginia Polytechnic and State University,  225 Stanger Street, Blacksburg, VA 24061-1026, USA
}
\email{josemadrid@vt.edu}

\subjclass[2020]{11B30, 11B13, }

\maketitle
\begin{abstract}
We show that for all $A, B \subseteq \{0,1,2\}^{d}$ we have $$
|A+B|\geq (|A||B|)^{\log(5)/(2\log(3))}.
$$
 We also show that for  all finite $A,B \subset \mathbb{Z}^{d}$, and any $V \subseteq\{0,1\}^{d}$ the inequality 
$$
|A+B+V|\geq |A|^{1/p}|B|^{1/q}|V|^{\log_{2}(p^{1/p}q^{1/q})}
$$
holds for all $p \in (1, \infty)$, where  $q=\frac{p}{p-1}$ is the conjugate exponent of $p$.  All the estimates are dimension free with the best possible exponents. We discuss applications to various related problems.


\end{abstract}

\section{Introduction}
\subsection{Minkowski sum for subsets of a discrete cube}
It is well-known that given arbitrary compact sets $A_{1}, \ldots, A_{n} \subset \mathbb{R}^{d}$ the classical Brunn--Minkowski inequality states that 
\begin{align}\label{brunnm}
\mu(A_{1}+\ldots+A_{n})\geq \left(\sum_{j=1}^{n} \mu(A_{j})^{1/d}\right)^{d}, 
\end{align}
where $d\mu$ is the Lebesgue measure on the product space  $\mathbb{R}^{d}$, and $\sum_{j=1}^{n}A_{j}$ denotes the Minkowski sum.  The Gaussian counterpart of the inequality (\ref{brunnm}), i.e., the so called Ehrhard's inequality states that
\begin{align}\label{ehrh}
\gamma(A_{1}+\ldots+A_{n})\geq \Phi\left(\sum_{j=1}^{n}\Phi^{-1}(\gamma(A_{j}))\right),
\end{align}
where $\gamma$ denotes the standard Gaussian measure on $\mathbb{R}^{d}$  with density $\frac{e^{-\|x\|^{2}/2}}{(2\pi)^{d/2}}dx$, $\|x\|^{2} = x_{1}^{1}+\ldots+x_{d}^{2}$ for $x=(x_{1}, \ldots, x_{d})\in \mathbb{R}^{d}$,  and $\Phi(t) = \int_{-\infty}^{t}\frac{e^{-s^{2}/2}}{\sqrt{2\pi}}ds$. Both estimates (\ref{brunnm}) and (\ref{ehrh}) are sharp (see \cite{IV} and references therein). 

However, less is known when $A_{j}$ are subsets of a discrete cube and the underlying measure is the uniform counting measure. One of the goals of this paper is to find the largest possible power $p_{n,m}$ such that the inequality 
\begin{align}\label{bur}
\left|\sum_{j=1}^{n}A_{j}\right|\geq \prod_{j=1}^{n}|A_{j}|^{p_{n,m}}
\end{align}
holds for all subsets $A_{1}, \ldots, A_{n} \subset \mathcal{C}_{m}^{d}:= \{0,1,  \ldots, m\}^{d}\subset \mathbb{Z}^{d}$, all $d\geq 1$,  where $n, m$ are fixed nonnegative integers, and $|\cdot |$ denotes the uniform counting measure of a discrete set.

Let us make a couple of observations. 

1. For arbitrary sets one has the following trivial bounds
\begin{equation}\label{trivial bounds}
\prod_{i=1}^{n}|A_i|\geq\left|\sum_{j=1}^{n}A_{j}\right|\geq \max_{1\leq i\leq n}\{|A_i|\}\geq \prod_{i=1}^{n}|A_i|^{1/n},
\end{equation}
which imply that $1\geq p_{n,m}\geq \frac{1}{n}.$

2. Choosing $A_{1}=\ldots=A_{n} = \mathcal{C}_{m}^{d}$ we obtain
\begin{align}\label{sharp01}
\frac{\log(nm+1)}{n\log(m+1)}\geq p_{n,m}\,.
\end{align} 

3. The main result in \cite{W} and independently in \cite{HS1} states that 
$p_{2,1}=\frac{\log(3)}{\log(4)}$, i.e., 
we have equality in (\ref{sharp01}) in the case $(n,m)=(2,1)$. The equality  $p_{n,1}=\frac{\log(n+1)}{n\log(2)}$ was conjectured in \cite{HS1}, and established\footnote{We are grateful to Ben Green for pointing to us Appendix B in \cite{BG2} which contains more comprehensive list of references on the value $p_{n,1}$.} in \cite{BRO} and independently in \cite{LLS1}. The proof in \cite{LLS1} uses conformal maps and Hadamard's three circle theorem, whereas the proof in \cite{BRO} is more direct and uses a clever inductive argument. We will present a new proof (Theorem~\ref{mth01}) of the equality $p_{n,1}=\frac{\log(n+1)}{n\log(2)}$ based on Descartes' rule of signs.

4. For $m\geq 3$ let $\tau_{m}$ be the solution of the equation 
\begin{align*}
1+(m(m+1))^{\tau_{m}}=(m+1)^{2\tau_{m}}.
\end{align*}
The lower bound $p_{2,m}\geq \tau_{m}$ was obtained in \cite{BDFKK}, and it was also asked whether we have equality in (\ref{sharp01}) for $n=2$. The behaviour of the power $\tau_{m}$, as $m$ tends to $\infty$, asymptotically matches the behaviour of the power $\frac{\log(2m+1)}{2\log(m+1)}$.

For any $p \geq 1$ and any $f :\mathbb{Z}^{d} \to \mathbb{R}$ with finite support set 
\begin{align*}
    \| f\|_{\ell^{p}(\mathbb{Z}^{d})} = \left(\sum_{x \in \mathbb{Z}^{d}} |f(x)|^{p} \right)^{1/p}.
\end{align*}
In what follows we will write $\| f\|_{p}$ instead of $\| f\|_{\ell^{p}(\mathbb{Z}^{d})}$ and it will be clear from the context what is the value for $d$. Given functions $f, g :\mathbb{Z}^{d} \to \mathbb{R}$ with finite support we define their sup convolution as follows 
\begin{align*}
    f\bar{*}g(z) = \sup_{x \in \mathbb{Z}^{d}}f(x)g(z-x). 
\end{align*}

In the next theorem we affirmatively answer the question asked in \cite{Kath} which also confirms the case  $m=2$ of the question in \cite{BDFKK}.  
\begin{theorem}\label{mth02}
For all $d\geq 1$, and all $f,g : \mathbb{Z}^{d} \to [0, \infty)$ supported on $\{0,1,2\}^{d}$, we have 
\begin{align}\label{functionf2}
    \| f \bar{*} g\|_{1} \geq  \|f\|_{q} \|g\|_{q}, \quad \text{where} \quad q=\frac{2\log(3)}{\log(5)}.
\end{align}
In particular, for all $A, B \subset \{0,1,2\}^{d}$, and all $d\geq 1$ we have 
\begin{align}\label{set01}
|A+B|\geq |A|^{p}|B|^{p}
\end{align}
holds with the best possible power $p=\frac{\log(5)}{2\log(3)}=0.732\ldots$. 
\end{theorem}

The inequality (\ref{set01}) follows from  (\ref{functionf2}) by taking indicator functions of sets i.e., $f = 1_{A}$ and $g=1_{B}$.

The following result was obtained by  Matolcsi, Ruzsa, Shakan and Zhelezov.
\begin{theorem}[\cite{MRSZ}]\label{mrsz1}
Let $d\geq 1$ and $A, B\subset \Z^d$ be finite sets and let $V\subseteq \{0,1\}^d$. Then 
\begin{align}\label{diskb}
|A+B+V|\geq \left(\frac{p^{1/p}q^{1/q}}{2}\right)^{d}|V||A|^{1/p}|B|^{1/q},
\end{align}
for any $p \in (1,\infty),$ where $q=\frac{p}{p-1}$ is the conjugate exponent of $p$. 
\end{theorem}

Notice that by choosing $V=\{0,1\}^{d}$ and $1/p = |A|^{1/d}(|A|^{1/d}+|B|^{1/d})^{-1}$ the estimate (\ref{diskb}) becomes
\begin{align}\label{diskbr}
|A+B+\{0,1\}^{d}|^{1/d}\geq |A|^{1/d}+|B|^{1/d}.
\end{align}
The inequality (\ref{diskbr}) can be considered as the {\em discrete} Brunn--Minkowski inequality as it implies the Brunn--Moinkowski inequality (\ref{brunnm}) for two sets (see \cite{MRSZ} and references therein). In this paper we obtain 
\begin{theorem}\label{trisets}
For all $d\geq 1$, and all finitely supported $f,g,h : \mathbb{Z}^{d} \to [0, \infty)$ with $f$ supported on $\{0,1\}^{d}$ we have 
\begin{align}\label{functionf3}
    \| f\bar{*} g \bar{*}h\|_{1} \geq \|f\|_{c} \|g\|_{p} \|h\|_{q}
\end{align}
holds for all $p \in (1, \infty)$, $q=\frac{p}{p-1}$ is the conjugate exponent of $p$, and $c$ satisfies $2^{1/c}=p^{1/p}q^{1/q}$. In particular,  we have 
    \begin{align}\label{trisets1}
    |V+A+B|\geq |V|^{1/c}|A|^{1/p}|B|^{1/q},
    \end{align}
for all finite subsets $A, B \subset \mathbb{Z}^{d}$ and $V \subset \{0,1\}^{d}$.
\end{theorem}

The theorem positively answers the question asked in \cite{Shakan1} (see the discussion right after Theorem 11.1 in \cite{MRSZ}, and see \cite{Shakan1}).  Notice that Theorem~\ref{trisets} implies Theorem~\ref{mrsz1} because of the trivial estimate 
\begin{align*}
|V|^{1/c} \geq \left(\frac{p^{1/p}q^{1/q}}{2}\right)^{d}|V|
\end{align*}
due to  $c\geq 1$ and  $|V|^{1-1/c} \leq |\{0,1\}^{d}|^{1-1/c}=\left(\frac{2}{p^{1/p}q^{1/q}}\right)^{d}$.

Theorem~\ref{mrsz1}, in the case $p=q=2$, is the key ingredient in obtaining query complexity for learning subsets $A \subset \mathbb{Z}^{d}$ having small {\em doubling constant}, see \cite{DZDP}. A short self-contained proof of Theorem~\ref{mrsz1} was recently given in \cite{GMRSZ} in the case $p=q=2$.

The next theorem due to Landau, Logan, Shepp \cite{LLS1} and independently Brown, Keane, Moran and Pierce \cite{BRO} states that we have equality in (\ref{sharp01}) for the pairs $(n,m)=(n,1)$ for all $n\geq 1$. We will present a new proof of it based on Decartes' rule of signs. 

\begin{theorem}[\cite{LLS1}, \cite{BRO}]\label{mth01}
For all $d\geq 1$, and all $f_{j} : \mathbb{Z}^{d} \to [0, \infty)$, $j=1, \ldots, n$, supported on $\{0,1\}^{d}$, we have 
\begin{align}\label{functionf}
    \| f_{1}\bar{*}\cdots \bar{*}f_{n}\|_{1} \geq \prod_{j=1}^{n} \|f_{j}\|_{q}, \quad \text{where} \quad q=\frac{n}{\log_{2}(n+1)}.
\end{align}
In particular, for all subsets $A_{j} \subseteq \{0,1\}^{d}$, we have 
\begin{align}\label{hyperc}
\left| \sum_{j=1}^{n}A_{j}\right| \geq \left(\prod_{j=1}^{n}|A_{j}|\right)^{p}
\end{align}
with the best possible power  $p=\frac{\log_{2}(n+1)}{n}$. 
\end{theorem}
 A related inequality to (\ref{functionf}) was established by Ivanisvili \cite[Theorem 1]{I}.

For any positive integer $k$ set
$$
kA:=\underbrace{A+A+\dots+A}_{k}. 
$$
Instead of considering subsets of cubes, one may consider sets containing cubes, that is the content of our next result.

\begin{theorem}\label{main thm 2}
Let $d\geq 1$. The following inequalities hold
   $$
 \left(k-1+\frac{1}{k}\right)^d\geq  \inf_{\{0,1,\dots,k-1\}^d\subseteq A\subset\Z^d}\frac{|kA|}{|A|}\geq k^{\frac{d(k-1)}{k}}.
   $$
\end{theorem}
\begin{rem}
    Observe that $\lim_{k\to\infty}\frac{(k-1+\frac{1}{k})^d}{k^{\frac{d(k-1)}{k}}}=1$, thus our bounds are asymptotically optimal.
\end{rem}
\begin{rem}
    Our Theorem \ref{main thm 2} can be seen as an extension of a result obtained by Green and Tao \cite[Proposition 4.1]{GT}, for $k = 2$.
\end{rem}

For any $\lambda>0$ set
$$
A+\lambda\cdot B=\{a+\lambda b;\ a\in A, b\in B\}.
$$

The {additive energy} $E(A)$ of a finite subset $A$ of an additive group $G$ is defined as 
$
E(A):=|\{(a_1,a_2,a_3,a_4)\in A^4;\ a_1+a_2=a_3+a_4\}|,
$
(see \cite{TaoVu}). It was proved by Kane and Tao \cite[Theorem 7]{KT} that $E(A)\leq |A|^{\log_26}$ for all $A\subseteq\{0,1\}^d$. 
For any $k\geq 2$, the {$k-$additive energy} $E_k(A)$ of a finite subset $A$ of an additive group $G$ is defined as 
$
E_k(A):=|\{(a_1,a_2,\dots,a_{2k})\in A^{2k};\ a_1+a_2+\dots+a_k=a_{k+1}+a_{k+2}+\dots+a_{2k}\}|.
$
The result of Kane-Tao was later extended to all $k\geq 2$, see \cite[Theorem 2]{DGIM} and \cite[Corollary 2]{K}, establishing that $E_{k}(A)\leq |A|^{\log_{2}{{2k}\choose{k}}}$ for all $A\subseteq\{0,1\}^d$. Combining this estimate with (\ref{hyperc}) we obtain 
\begin{align*}
    E_k(A)&\leq |A|^{\log_2{{2k}\choose{k}}}=|A|^{(\log_{2}{\frac{k+1}{2}})(\log_{\frac{k+1}{2}}{{2k}\choose{k}})}\\
    &=|A|^{(\log_{2}{(k+1)}-1)(\log_{\frac{k+1}{2}}{{2k}\choose{k}})}\stackrel{(\ref{hyperc})}{\leq} 
    \left(\frac{|kA|}{|A|}\right)^{\log_{\frac{k+1}{2}}{{2k}\choose{k}}}.
\end{align*}
The upper bound on $E_{k}(A)$ in terms of $|kA|$ gives a relation between additive energies and doubling constants. 
\begin{corollary}
Let $d\geq 1$, $k\geq 2$, and $p=\log_{\frac{k+1}{2}}{{2k}\choose{k}}$. We have 
$$
E_k(A)\leq \left(\frac{|kA|}{|A|}\right)^{p}
$$
for all $A\subseteq \{0,1\}^d$. Moreover, $p$ can not be replaced by a smaller quantity. In particular,
$
E(A)\leq \left(|A+A|/|A|\right)^{\log_{3/2}6}
$
for all $A\subseteq \{0,1\}^d$. 
\end{corollary}

The sharpness of the power $p$ in the corollary follows by testing the inequality for $A=\{0,1\}^{d}$.

In \cite{HP} Hanson and Petridis study sum of dilates, they proved that if $A$ is a subset of a commutative group $G$ and $K$ is a parameter such that $|A+A|\leq K|A|$, then $|A+2\cdot A|\leq K^{2.95}|A|$.

As a consequence of the simple inequality $\frac{|A+2\cdot A|}{|A|}\leq |A|$ and (\ref{hyperc}) one has
\begin{corollary}\label{Corollary HP}
Let $d\geq 1$ and $p=\log_{3/2}2\approx 1.7095$. The inequality
$$
|A+2\cdot A|\leq \left(\frac{|A+A|}{|A|}\right)^{p}|A|
$$
holds for all $A\subseteq\{0,1\}^d$. Moreover, $p$ cannot be replaced by a smaller power.
\end{corollary}

The sharpness of the power  $p=\log_{3/2}2$ follows by considering $A=\{0,1\}^d$. 

Next we study a variant of the distinct distances problem posed by Erd\H os. For positive integers $N,k,\ell,m$, we consider sets of $N$ points in $\mathbb{R}^2$ with the property that every $k$ of the points span at least $\ell$ distinct $m$-sums.
Let $\psi_{m}(N,k,\ell)$ denote the minimum number of $m-$sums such a point set can span.

\begin{corollary}\label{Geom application thm}
For every $N\geq k > 0$ and $m\geq 1$, where $N$ is the power of two,  we have 
$$ 
\psi_{m}\left(N,k,k^{\log_{2} (m+1)}\right)=O\left(N^{\log_{2} (m+1)}\right).
$$
\end{corollary}
This can be seen as an extension of the main theorem in \cite{FLS}. As in \cite{FLS} the hypercube $\{0, 1\}^{n} \subset \mathbb{R}^{n}$ can be projected (which is also bijection onto its image) by a transformation $T$ onto the real line $\mathbb{R}$ giving a set $P_{n}$ of cardinality $N=2^{n}$  so that any subset of $\tilde{A}\subset P_{n}$ has the same number of distinct m-sums as the set $A\subset\{0,1\}^{n}$, where  $T(A)=\tilde{A}$. Thus the corollary is a direct consequence of (\ref{hyperc}). The projection is defined as follows: Let $n\in\mathbb{Z}^{+}$. Choose  
$r_1,r_2,\dots,r_n\in\mathbb{R}\setminus\mathbb{Q}$ such that $r_i> 0$ for all $1\leq i\leq n$ and the only integer solution to
$x_1\cdot r_1 + x_2\cdot r_2 + \cdots + x_{n} \cdot r_{n} = 0,$ 
is $x_1=x_2=\cdots=x_{n}=0$. For every integer $0\le j \le n$, we define a set $P_j$ of $2^j$ real numbers, as follows.
We set $P_0 = \{1\}$.
For $1\le j \le n$, we define $P_j = P_{j-1} \bigcup (r_j + P_{j-1})$.

In the next section we present the proofs of the main results.

\section{Proof of the main results}

The next theorem shows the functional inequalities (\ref{functionf}), (\ref{functionf2}) and (\ref{functionf3}) hold for all dimensions $d\geq 1$  if and only if they hold in dimension $d=1$. 
\begin{theorem}[Compressing dimension]\label{mth001} Fix positive numbers $p_{1}, \ldots, p_{n}>0$ and nonempty finite sets $S_{1}, \ldots, S_{n} \subset \mathbb{Z}$. The inequality 
\begin{align}\label{fbur}
 \sum_{z \in \mathbb{Z}^{d}}\max_{x^{(1)}+\ldots+x^{(n)}=z} \prod_{j=1}^{n}f^{p_{j}}_{j}(x^{(j)})  \geq \prod_{j=1}^{n}\left(\sum_{x \in \mathbb{Z}^{d}}f_{j}(x)\right)^{p_{j}}
\end{align}
holds for all non-negative functions $f_{j}$ supported on $S_{j}^{d}$, $j=1, \ldots, n$,   and all $d\geq 1$,  if and only if it holds for $d=1$. 
\end{theorem}
\begin{proof}
Suppose (\ref{fbur}) holds for $d=1$. For each $j=1, \ldots, n,$ let $x^{(j)}=(\overline{x}^{(j)}, \tilde{x}^{(j)})\in \mathbb{Z}^{d}$, where $\overline{x}^{(j)}\in \mathbb{Z}^{d-1}$. Similarly let $z=(\overline{z}, \tilde{z}) \in \mathbb{Z}^{d}$. Set $\overline{f_{j}}(\overline{x}^{(j)}) = \sum_{\tilde{x}^{(j)}\in \mathbb{Z}} f_{j}(x^{(j)})$ for each $j=1, \ldots, n$. We have 
\begin{align*}
   &\sum_{z \in \mathbb{Z}^{d}}\max_{x^{(1)}+\ldots+x^{(n)}=z} \prod_{j=1}^{n}f^{p_{j}}_{j}(x^{(j)})  = \\
   & \sum_{\overline{z} \in \mathbb{Z}^{d-1}}\sum_{\tilde{z} \in \mathbb{Z}} \max_{\overline{x}^{(1)}+\ldots+\overline{x}^{(n)}=\overline{z}} \max_{\tilde{x}^{(1)}+\ldots+\tilde{x}^{(n)}=\tilde{z}}\prod_{j=1}^{n}f^{p_{j}}_{j}(\overline{x}^{(j)}, \tilde{x}^{(j)})\geq \\
   & \sum_{\overline{z} \in \mathbb{Z}^{d-1}} \max_{\overline{x}^{(1)}+\ldots+\overline{x}^{(n)}=\overline{z}} \sum_{\tilde{z} \in \mathbb{Z}} \max_{\tilde{x}^{(1)}+\ldots+\tilde{x}^{(n)}=\tilde{z}}\prod_{j=1}^{n}f^{p_{j}}_{j}(\overline{x}^{(j)}, \tilde{x}^{(j)})\stackrel{\text{case}\, d=1}{\geq} \\
   &\sum_{\overline{z} \in \mathbb{Z}^{d-1}} \max_{\overline{x}^{(1)}+\ldots+\overline{x}^{(n)}=\overline{z}} \prod_{j=1}^{n}\overline{f}^{p_{j}}_{j}(\overline{x}^{(j)})\stackrel{\text{iterate}}{\geq}\ldots \geq \prod_{j=1}^{n}\left(\sum_{x \in \mathbb{Z}^{d}}f_{j}(x)\right)^{p_{j}}.
\end{align*}
\end{proof}


In the case $d=1$, $S_{1}=\ldots=S_{n}=\{0,1, \ldots, m\}$, after a change of variables, the inequality (\ref{fbur}) can be rewritten as follows: for any non-negative $y_{i}^{(j)}$ with $i=0,\ldots, m$, and $j=1, \ldots, n,$ we have 
\begin{align}\label{eq:general-inequality}
        \sum_{k=0}^{nm}\max_{i_1+\dots+i_n=k} y_{i_1}^{(1)}y_{i_2}^{(2)}\dots y_{i_n}^{(n)}
        \geq 
        \prod_{j=1}^n \left(\sum_{i=0}^{m} (y_{i}^{(j)})^{1/p_{j}}\right)^{p_{j}}.
\end{align}

It turns out that to verify (\ref{eq:general-inequality}) for some $p_{1}, \ldots, p_{n}>0$ it suffices to consider the case when for each $j=1,\dots, n,$ the sequence  $\{y_i^{(j)}\}_{i=0}^{m}$ is non-increasing. 

\begin{lemma}\label{perestanovka}
  Let $y_i^{(j)}$ with $i=0,1,\dots, m$ and $j=1,2,\dots, n$ be non-negative real numbers. Then if for each $j=1,\dots, n$ we sort all $x_i^{(j)}$ from the largest to the smallest, the value of the left hand side of \eqref{eq:general-inequality} will not increase.
\end{lemma}
\begin{proof}
    It suffices \footnote{Indeed, this follows from the fact that the sum of numbers can be written as an integral of the cardinality of the set of indices corresponding to numbers at least $t$.} to show that for any $t\geq 0$ the cardinality of the set 
    \[
    \mathcal{K}(t):=\{k: \, \exists i_1+\dots+ i_n =k,\, x_{i_1}^{(1)}x_{i_2}^{(2)}\dots x_{i_n}^{(n)} \geq t \} \subset \{0, 1, \dots, nm+1\}
    \]
    will not increase. This can be written as an infinite union 
    \[
    \cup_{t_1\dots t_n=t} \{i_1: x_{i_1}^{(1)} \geq t_1\}+\dots +\{i_n: x_{i_n}^{(n)} \geq t_n\},
    \]
    where the sum is understood as a sum of subsets of $\mathbb{Z}$. Rearranging $x_i^{(j)}$'s for each $j$ does not change the cardinality of the sets involved. After the sort each of the sets becomes an interval of the form $[0, L)$ and so the union of all these sums is equal to its maximal value which, in turn, is equal to 
    \[
\max_{t_1\dots t_n=t} 1-n+\sum_{j=1}^n \left|\left\{x_{i_j}^{(j)} \geq t_j\right \}\right|,
    \]
    which is a lower bound for the original value as follows from the Cauchy-Davenport theorem.
\end{proof}

\begin{rem}\label{genrem1}
To prove Theorem~\ref{mth01},  Theorem~\ref{mth02}, and Theorem~\ref{trisets} it suffices to verify (\ref{fbur}) in the case $d=1$, A) $S_{1}=S_{2}=\ldots=S_{n}=\{0,1\}$, $p_{1}=\ldots=p_{n}=\frac{\log_{2}(n+1)}{n}$; B) $n=2$, $S_{1}=S_{2}=\{0,1,2\}$, $p_{1}=p_{2}=\frac{\log(5)}{2\log(3)}$; C) $n=3$, $S_{1}=\{0,1\}$, $S_{2}, S_{3} \subset \mathbb{Z}$ are arbitrary finite, $p_{1}=1/c, p_{2}=1/p, p_{3}=1/q$. 
\end{rem}

\subsection{The proof of Theorem~\ref{mth01}}

Our goal is to prove the following inequality
\begin{align}\label{npoint}
 \sum_{k=0}^{n}\left(\max_{i_1+\dots+i_n=k} y_{i_1}^{(1)}y_{i_2}^{(2)}\dots y_{i_n}^{(n)}\right)^p 
        \geq 
        \prod_{j=1}^n \left( y_{0}^{(j)}+y_{1}^{(j)}\right)^p.
\end{align}
By homogeneity we can assume that for each $j=1, \ldots, n$ we have  $y_{0}^{(j)}+y_{1}^{(j)}=1$. Also without loss of generality we can assume $y_{0}^{(j)}\geq y_{0}^{(j+1)}$ for $j=1, \ldots, n-1$, and $y_{0}^{(1)}\geq 1/2$. It will be convenient to denote $x_{i,0}=y^{(i)}_{0}$ for all $i=1, \ldots, n$, and also set $q=pn$. The desired inequality takes the form
\begin{align}\label{abakidev}
\prod_{i=1}^{n}x^{q/n}_{i,0}+\prod_{i=1}^{n}(1-x_{i,0})^{q/n}+
    \sum_{\ell=1}^{n-1}\prod_{i=1}^{\ell}x^{q/n}_{i,0}\prod_{i=\ell+1}^{n} (1-x_{i,0})^{q/n} \geq 1. 
\end{align}

For each $i=1, \ldots, n$, we define
$$
w_{i}=\left[\frac{1-x_{n+1-i,0}}{x_{n+1-i,0}}\right]^{q/n}.
$$
The inequality (\ref{abakidev}) reduces to verify that
\begin{equation}\label{key inequality for Thm 1}
1+w_1+w_1w_2+w_1w_2w_3+\dots+\prod_{i=1}^{n}w_i\geq \prod_{i=1}^{n}(1+w^{n/q}_i)^{q/n}.
\end{equation}
Observe that since  $\{x_{i,0}\}_{i=1}^{n}$ is nonincreasing, then $\{w_i\}_{i=1}^{n}$ is also nonincreasing. The inequality $\eqref{key inequality for Thm 1}$ is a consequence of Lemma~\ref{Thm 1 key lemma 1} and Lemma~\ref{Thm 1 key lemma 2}.

\begin{lemma}\label{Thm 1 key lemma 1}
We have 
$$
\left(1+w_1+w_1w_2+w_1w_2w_3+\dots+\prod_{i=1}^{n}w_i\right)^n\geq \prod_{i=1}^{n}\left(\sum_{k=0}^{n}w^k_i\right).
$$
for any nonnegative numbers $w_{1}, \ldots, w_{n},$ with $w_1\geq w_2\geq\ldots\geq w_n$. 
\end{lemma}
For each positive integer $k$ consider  $f_k(w)=(1+w+\dots+w^k)^{1/k} $ on $\mathbb{R}_+$. We need the following technical statement. 

\begin{lemma}\label{lem:increasing-ratio}
     For every $n\geq 2$ let $g_n(w)=f_{n-1}(w)/f_{n}(w)$. Then $g_n$ satisfies $g_n(1/w)=g_n(w)$, is increasing on $[0, 1]$, and decreasing on $[1,\infty)$.
\end{lemma}
\begin{proof}[Proof of Lemma \ref{lem:increasing-ratio}]
    The equality $g_n(w)=g_n(1/w)$ follows from the fact that $f_k(1/w)=f_k(w)/w$ for positive integers $k$. By the $w\rightarrow 1/w$ symmetry it suffices to prove that the function is decreasing for $w\geq 1$. Taking logarithmic derivative we observe that $(\ln(g_n(w)))'\leq 0$ is equivalent to have the inequality
    \[
    \frac{1}{n}\cdot \frac{1+2w+\dots+nw^{n-1}}{1+w+\dots+w^{n}}
    \geq 
    \frac{1}{n-1}\cdot \frac{1+2w+\dots+(n-1)w^{n-2}}{1+w+\dots+w^{n-1}}.
    \]
    After multiplication by denominators and substraction of the common part, i.e. subtracting $(n-1)(1+2w+\dots+(n-1)w^{n-2})(1+w+\dots+w^{n-1})$, we get
    \begin{align*}
    &(n-1)nw^{n-1}(1+w+\dots+w^{n-1})\\
    &\geq (1+w+\dots+w^{n-1})(1+2w+\dots+(n-1)w^{n-2})
    \\&\ \ \ \ +
    nw^n(1+2w+\dots+(n-1)w^{n-2}).
    \end{align*}
    Expanding both sides and subtracting we arrive at the inequality
    \[
    \sum_{k=1}^{n-1} \frac{k(k+1)}{2}\cdot w^{2n-2-k} \geq \sum_{k=1}^{n-1} \frac{k(k+1)}{2}\cdot w^{k-1}, 
    \]
    which is trivial for $w\geq 1$.
\end{proof}

\begin{proof}[Proof of Lemma \ref{Thm 1 key lemma 1}]
 We take the $n$-th root of both sides, so the inequality gets rewritten as
 $$
 1+w_1+w_1w_2+\dots+w_1w_2\dots w_n\geq \prod_{i=1}^{n} f_n(w_i).
 $$
 We will use induction on $n$, the case $n=1$ is trivial. Take $n\geq 2$ and assume the inequality is known for $n-1$ variables. Note that the inequality is invariant under the change $(w_1,\dots, w_n) \rightarrow (1/w_n,\dots, 1/w_1)$. So making this change, if needed, we may assume that $w_1w_n\geq 1$, which we do from now on. 

First, we apply the inequality to $(w_2,\dots,w_n)$ and write $w$ in place of $w_1$, so that the inequality is reduced to the following: 
\[
1+w\cdot \prod_{i=2}^n f_{n-1}(w_i) \geq f_n(w)\cdot \prod_{i=2}^n f_n(w_i).
\]
Now, we write $w$ as $(f_n(w)^n - 1)/f_{n-1}(w)^{n-1}$ and divide both sides by $f_n(w)^n$ to arrive at an equivalent inequality 
\[
\frac{1}{f_n(w)^n} + \left(1-\frac{1}{f_n(w)^n}\right)\cdot \prod_{i=2}^n \frac{f_{n-1}(w_i)}{f_{n-1}(w)} \geq \prod_{i=2}^n \frac{f_n(w_i)}{f_n(w)}.
\]
Note that since both $f_{n-1}$ and $f_n$ are increasing and $w\geq w_i$, both products are at most $1$, and since the left hand side of the inequality above is a convex combination of the product and 1, it suffices to prove that 
\[
\prod_{i=2}^n \frac{f_{n-1}(w_i)}{f_{n-1}(w)} \geq \prod_{i=2}^n \frac{f_n(w_i)}{f_n(w)},
\]
which would follow from the inequality term-wise, which, in turn, can be rewritten as 
\[
\frac{f_{n-1}(w_i)}{f_n(w_i)} \geq \frac{f_{n-1}(w)}{f_n(w)}.
\]
This is true since we have $ww_n\geq 1$, which implies that $w\geq 1/w_n\geq 1/w_i$ and so taking $y:=\max\{w_i, 1/w_i\}$ we have $1\leq y\leq w$ and by Lemma \ref{lem:increasing-ratio}
\[
\frac{f_{n-1}(w_i)}{f_n(w_i)} = \frac{f_{n-1}(y)}{f_n(y)}\geq \frac{f_{n-1}(w)}{f_n(w)}.
\]
\end{proof}


\begin{lemma}\label{Thm 1 key lemma 2}
Let $r=\log_{2}(n+1)$. We have 
\begin{equation*}
\sum_{i=0}^{n}w^i-(1+w^{n/r})^{r}\geq 0
\end{equation*}
for all $w\geq 0$.
\end{lemma}

We will make a general observation which implies Lemma~\ref{Thm 1 key lemma 2}, and which also will be used later  in proofs of other technical statements. The general observation says that if we consider the sequence of polynomials 
\begin{align*}
    P_{n}(x) = 1+x+x^{2}+\ldots+x^{n}, 
\end{align*}
then the functions $(P_{n}(x))^{1/n}$ and $(P_{m}(x))^{1/m}$ behave similarly at infinity. We can match their values at $x=1$ by replacing $(P_{m}(x))^{1/m}$ with $P_{m}(x^{s})^{1/(ms)},$ where $s=\frac{n\log(m+1)}{m\log(n+1)}$ solves the equation $(P_{n}(1))^{1/n}=P_{m}(1)^{1/(ms)}$. Then it turns out that $(P_{n}(x))^{1/n} \geq P_{m}(x^{s})^{1/(ms)}$ for all $x\geq 0$ if and only if $n\geq m$. The inequality $(P_{n}(x))^{1/n}\geq (P_{1}(x^{s}))^{1/s}$ is the same as in Lemma~\ref{Thm 1 key lemma 2}. The general observation follows from the following 

\begin{lemma}\label{lem:main}
    For any $n > m\geq 1$ and $x>0$ one has 
    $F_n(x)\geq F_m(x)$, where 
    $$
    F_{n}(x):=
P_n\left(x^{\tfrac{\log{(n+1)}}{n}}\right)^{\tfrac{1}{\log{(n+1)}}}.
    $$
\end{lemma}
In order to prove the lemma we use the following 
\begin{proposition}\label{prop:properties}
    For every $n\geq 1$, the function $F_n$ satisfies $F_n(0)=1, F_n(1)=e, F_n(1/x)=F_n(x)/x, F_n'(1)=e/2$.
\end{proposition}
\begin{proof}
   First three assertions are obvious and the last one follows by differentiating the functional equation and substituting $x=1$. 
\end{proof}
\begin{proof}[Proof of Lemma \ref{lem:main}]
Without loss of generality we may assume that $n=m+1$. By Proposition \ref{prop:properties} the equation $F_n(x)=F_m(x)$ has a root at $x=0$, a double root at $x=1$ and all other roots come in pairs $(x, 1/x)$. Since close to zero we have $F_n(x)\sim 1+x^{\tfrac{\log{(n+1)}}{n}}/\log{(n+1)}$, we know that $F_n(x) \geq F_m(x)$ in some neighbourhood of zero. Hence, if the statement of the lemma is false, we must have at least one $t\in (0, 1)$ for which $F_n(t)=F_m(t)$. Seeking a contradiction we assume that $F_n(t)=F_m(t)$ for some $t\in (0, 1)$. 


Let us compute the logarithmic derivative of $F_n$.

\begin{align*}
 \partial_x \log F_n(x) &= 
 \frac{1}{\log{(n+1)}}\times\\
 &\left(-\frac{\tfrac{(n+1)\log{(n+1)}}{n}\cdot x^{\tfrac{(n+1)\log{(n+1)}}{n}-1}}{1-x^{\tfrac{(n+1)\log{(n+1)}}{n}}} + \frac{\tfrac{\log{(n+1)}}{n}\cdot x^{\tfrac{\log{(n+1)}}{n}-1}}{1-x^{\tfrac{\log{(n+1)}}{n}}}\right)  
 \\&=
 \frac{1}{nx}\cdot \left(-\frac{(n+1)\cdot x^{\tfrac{(n+1)\log{(n+1)}}{n}}}{1-x^{\tfrac{(n+1)\log{(n+1)}}{n}}} + \frac{x^{\tfrac{\log{(n+1)}}{n}}}{1-x^{\tfrac{\log{(n+1)}}{n}}}\right) 
\\&=
 \frac{1}{nx}\cdot \left(-(n+1)\cdot \left[\frac{1}{1-x^{\tfrac{(n+1)\log{(n+1)}}{n}}}-1 \right]  \right.\\
 &+\left. \left[\frac{1}{1-x^{\tfrac{\log{(n+1)}}{n}}}-1\right] \right) 
 \\&=
 \frac{1}{x}+\frac{1}{x}\cdot \left(-\frac{(n+1)/n}{1-x^{\tfrac{(n+1)\log{(n+1)}}{n}}}  + \frac{1/n}{1-x^{\tfrac{\log{(n+1)}}{n}}} \right) 
 \\&=
 \frac{1}{x}+\frac{1}{x}\cdot \frac{-1+\tfrac{n+1}{n}\cdot x^{\beta_n}-\tfrac{1}{n}\cdot x^{\alpha_n}}{(1-x^{\alpha_n})(1-x^{\beta_n})},
 \end{align*}
 where 
\begin{align*}
\alpha_n := \tfrac{(n+1)\log{(n+1)}}{n},\quad \text{and \quad }\beta_n := \tfrac{\log{(n+1)}}{n}.
\end{align*}
We now consider the function
\begin{align*}
&H_{n, m}(x):=\\
&x\cdot (\partial_x \log F_n(x) - \partial_x \log F_m(x))\cdot (1-x^{\alpha_n})(1-x^{\beta_n})(1-x^{\alpha_m})(1-x^{\beta_m}),
\end{align*}
which can be written as a {\em pseudo-polynomial}
\begin{align*}
&H_{n, m}(x)=(1-x^{\alpha_m})(1-x^{\beta_m})\left(-1+\tfrac{n+1}{n}\cdot x^{\beta_n}-\tfrac{1}{n}\cdot x^{\alpha_n}\right)
-\\
&(1-x^{\alpha_n})(1-x^{\beta_n})
\left(-1+\tfrac{m+1}{m}\cdot x^{\beta_m}-\tfrac{1}{m}\cdot x^{\alpha_m}\right).
\end{align*}




Expanding $H_{n, m}(x)$ as a pseudo-polynomial in the increasing order of exponents, we get 
\begin{align*}
    H_{n, m}(x) &= \tfrac{1}{n}x^{\beta_n} -\tfrac{1}{m} x^{\beta_m} + \left(\tfrac{m+1}{m}-\tfrac{n+1}{n}\right)
    x^{\beta_m+\beta_n} + \tfrac{m+1}{m} x^{\alpha_m} - \tfrac{n+1}{n}  x^{\alpha_n}\\
    &-\left(\tfrac{n+1}{n}+\tfrac{1}{m}\right)
x^{\beta_n+\alpha_m}- x^{\alpha_m+\beta_m}
+
x^{\alpha_n+\beta_n}
+\left(\tfrac{1}{n} + \tfrac{m+1}{m}\right)
x^{\beta_m+\alpha_n} \\
&+\tfrac{n+1}{n}
x^{\alpha_m+\beta_m+\beta_n}-\tfrac{m+1}{m}
x^{\alpha_n+\beta_n+\beta_m}
-\left(\tfrac{1}{m}-\tfrac{1}{n}\right)
x^{\alpha_m+\alpha_n} \\
&+\tfrac{1}{m}
x^{\alpha_n+\beta_n+\alpha_m}-\tfrac{1}{n}
x^{\alpha_m+\beta_m+\alpha_n}.
\end{align*}

To see that exponents are increasing, note that all exponents split into pairs with sum $\alpha_m+\beta_m+\alpha_n+\beta_n$, so it suffices to show that 
\[
\beta_n < \beta_m<\beta_m+\beta_n< \alpha_m< \alpha_n<\beta_n+\alpha_m<\alpha_m+\beta_m<\alpha_n+\beta_n.
\]
Inequalities 
\[
\beta_n < \beta_m<\beta_m+\beta_n< \alpha_m< \beta_n+\alpha_m<\alpha_m+\beta_m<\alpha_n+\beta_n,
\]
are satisfied for any $n>m$, as is the trivial inequality $\alpha_n > \alpha_m$. To verify that $\alpha_n < \beta_n+\alpha_m$ we recall our assumption that $n=m+1$ and so 
\[
\alpha_m = \tfrac{(m+1)\log{(m+1)}}{m} > \log{(m+2)} = \log{(n+1)} = \alpha_n-\beta_n,
\]
which implies the desired inequality. Here, the inequality in the middle is equivalent to $(m+2)^m < (m+1)^{m+1}$ which is true for $m\geq 1$.

So the signs of the monomials are $+-++---+++--+-$, making it 7 sign changes. Hence, by the Descartes' rule of signs for pseudo-polynomials (see example \#77 on page 46 in  \cite{Polya}) we conclude that $H_{n, m}$ has at most 7 positive roots counted with multiplicities. Since the product $(1-x^{\alpha_n})(1-x^{\beta_n})(1-x^{\alpha_m})(1-x^{\beta_m})$ already has root 1 of multiplicity 4, we conclude that 
\[
\partial_x\log F_n(x) - \partial_x\log F_m(x) = \partial_x(\log F_n(x) - \log F_m(x))
\]
must have at most 3 positive roots counted with multiplicities. 

Since $F_n(0)=F_m(0)$ and $F_n(1)=F_m(1)$, and, by our assumption, $F_n(t)=F_m(t)$ for some $t\in (0, 1)$, we conclude by Rolle's theorem that $\partial_x(\log F_n(x) - \log F_m(x))$ has a root in each of $(0, t)$ and $(t, 1)$. By the functional equation $F_n(1/x)=F_n(x)/x$ we have $F_n(1/t)=F_m(1/t)$ and so $\partial_x(\log F_n(x) - \log F_m(x))$ must also have a root in $(1, 1/t)$, again by Rolle's theorem. Hence, together with $x=1$ the function $\partial_x(\log F_n(x) - \log F_m(x))$ has at least 4 positive roots leading to a contradiction. 
\end{proof}



\subsection{The proof of Theorem~\ref{mth02}}
Using Remark~\ref{genrem1} it suffices to verify the following six-point inequality. 
\begin{proposition}
    \label{prop: prek lein m2}
    For all $x_0, x_1, x_2 \ge 0$ and $y_0, y_1, y_2 \ge 0$, it holds that 
    \begin{equation}
    \label{ineq: main m2}
        \sum_{k = 0}^4 \max_{i + j = k} x_i^py_j^p \ge (x_0 + x_1 + x_2)^p(y_0 + y_1 + y_2)^p\,.
    \end{equation}
\end{proposition}

We now prove Proposition \ref{prop: prek lein m2}. We proceed by successively reducing the number of variables, deducing the proposition from the following two lemmas.

\begin{lemma}
    \label{lem: 1var}
    The following inequalities hold for all $x \ge 0$
    \begin{equation}
        \label{ineq: m0}
        1 + x^p \ge (1 + x)^p
    \end{equation}
    \begin{equation}
        \label{ineq: m1}
        1 + x^p + x^{2p} \ge (1 + x)^{2p}
    \end{equation}
    \begin{equation}
        \label{ineq: m2}
        1 + x^p + x^{2p} + x^{3p} + x^{4p} \ge (1 + x + x^2)^{2p}
    \end{equation}
    \begin{equation}
        \label{ineq: m2b}
        1 + x^p + x^{2p} + x^{3p} \ge (1 + x)^p(1 + x + x^2)^p\,.
    \end{equation}
\end{lemma}

\begin{proof}
    Inequality \eqref{ineq: m0} holds since $p < 1$. Inequalities \eqref{ineq: m1} and \eqref{ineq: m2} follow from Lemma~\ref{Thm 1 key lemma 2}. We now prove the inequality (\ref{ineq: m2b}). Near $0$, the sharper inequality
    \begin{equation}
        \label{ineq: m2b sharper}
        1 + x^p \ge (1 + x)^p(1 + x + x^2)^p
    \end{equation}
    holds. Indeed, taking logarithmic derivatives of both sides and taking the difference yields
    $$
        \frac{1}{p} \partial_x \log \left( \frac{1+x^p}{(1 + x)^p(1 + x + x^2)^p}\right)=\frac{x^{p-1}}{1  +x^{p}} - \frac{1}{1+x} - \frac{1+2x}{1+x+x^2}\,.
    $$
    The sign of the logarithmic derivative coincides (after clearing all the denominators) with the sign of 
    $$
        x^{p-1} - 2- 4x - 2x^{1+p} - 3x^2 - 2x^{2+p}\,.
    $$
    This clearly changes sign exactly once, from positive to negative, and hence \eqref{ineq: m2b sharper} holds on an interval if it holds at the endpoints. By verifying the inequality at $x=0.1$ we find that it holds in particular on $[0,0.1]$. 
    
    It remains to show \eqref{ineq: m2b} on $[0.1,1]$. 
    From Lemma~\ref{Thm 1 key lemma 2}, we have that 
    $$
        1 + x^p + x^{2p} + x^{3p} \ge (1 + x^{3p/2})^2\,,
    $$
    so it is enough to show that for $x \in [0.1,1]$
    \begin{equation}
        \label{ineq: bound 01}
        (1 + x^{3p/2})^2 \ge (1 + x)^p(1+x+x^2)^p\,.
    \end{equation}
    Again, we compute logarithmic derivatives to obtain
    $$
        \frac{1}{p}\partial_x\log \left( \frac{(1+x^{3p/2})^2}{(1 + x)^p(1 + x + x^2)^p}\right) = 3 \frac{x^{3p/2 - 1}}{1 + x^{3p/2}} - \frac{1}{1+x} - \frac{1+2x}{1+x+x^2}\,.
    $$
    After clearing denominators, we see that the sign of the logarithmic derivative coincides with the sign of
    \begin{equation}
        \label{eq log der}
        -2 + 3x^{3p/2-1} - 4x + 4x^{3p/2}-3x^2 + 2x^{3p/2 +1}=:P(x)\,.
    \end{equation}
    $P$ has at most $5$ zeros on $(0,\infty)$, since this is the number of sign changes in the coefficients. Furthermore $P(x) = -x^{3p/2+1}P(1/x)$, so $P(1) = 0$ and the number of zeros of $P$ in $(0,1)$ is the same as the number of zeros in $(1,\infty)$. Clearly $P$ is negative near $0$, and one computes that $P'(1) < 0$. Thus, if there is more than one zero in $(0, 1)$, then there are at least $3$, so at least $7$ in total, which is impossible. So inequality \eqref{ineq: bound 01} holds if it holds in $0.1$ and in $1$, which it does.
\end{proof}

\begin{lemma}
    \label{lem: 2var}
    The following three inequalities hold for all $0 \le x \le y$:
    \begin{equation}
        \label{ineq: xeqy1}
        1 + y^p + y^{2p} + x^p y^{2p} + x^{2p}y^{2p} \ge (1 + y + xy)^{2p}
    \end{equation}
    \begin{equation}
        \label{ineq: xeqy2}
        1 + x^p + x^{p}y^p + x^{2p}y^p + x^{2p}y^{2p} \ge (1 + x + xy)^{2p}
    \end{equation}
    \begin{equation}
        \label{ineq: xneqy}
        1 + y^p + x^p y^p + x^{p}y^{2p} + x^{2p}y^{2p} \ge (1 + x + xy)^p(1 + y + xy)^p
    \end{equation}
\end{lemma}

\begin{proof}
    \noindent\textbf{Proof of inequality \eqref{ineq: xeqy1}:} 
    We note that the function  $f(x) := x^p - x^{p - 1} - 2x^{2p - 1}$ is increasing on $[0, \infty)$: Its derivative 
    $$
        \partial_x f(x) = p x^{p - 1} + (1 - p) x^{p-2} - 2 (2p - 1) x^{2p - 2}
    $$
    is positive by the weighted AM-GM inequality
    $$
        p x + (1- p) \ge x^p  > 2(2p - 1) x^p\,,
    $$
    where we use that $1/2 < p < 3/4$. 
    
    Now we consider $h = 1 + y + xy$ to be fixed, and show that under this constraint the left hand side of \eqref{ineq: xeqy1} is minimized at $x = y$ or $x =0$. To show this, it suffices to prove that its derivative with respect to $y$ changes sign at most once from plus to minus.  At the two endpoints $x = y$ and $x = 0$ the inequality holds by Lemma \ref{lem: 1var}, so this proves the inequality.

    We have $\partial_y x = -\frac{1 + x}{y}$, from which we obtain:
    \begin{equation}
        \frac{y^{1 - 2p}}{p} \partial_y(1 + y^p + y^{2p} + x^p y^{2p} + x^{2p}y^{2p})= y^{-p} + 2 + x^p  -x^{p-1} - 2x^{2p-1}\,.
    \end{equation}
    Thus the derivative has the same sign as $y^{-p} + 2 + f(x) = y^{-p} + 2 + f(\frac{h-1}{y+1})$, which is decreasing in $y$.
    \medskip

    \noindent\textbf{Proof of inequality \eqref{ineq: xeqy2}:} 
    We set $z = xy$. Then we have to show that for $0 \le x^2 \le z$, we have 
    $$
        1 + z^p + z^{2p} + (1 + z^p)x^p \ge (1 + x + z)^{2p}\,.
    $$
    At the endpoints $x = 0$ and $x^2 = z$, this holds by Lemma \ref{lem: 1var}. Near $0$, the difference of left hand side and right hand side has derivative $\sim p(1 + z^p)x^{p-1}$, which is positive. So it is enough to show that the logarithmic derivative of the ratio of left hand side and right hand side vanishes at most once. It is 
    $$
        \frac{p(1+z^p)x^{p-1}}{1 + z^p + z^{2p} + (1 + z^p)x^p} - 2p \frac{1}{1 + x + z}\,.
    $$
    Multiplying through, it suffices to show that 
    \begin{align*}
        &\quad p(1+z^p)x^{p-1}(1 + z + x) - 2p(1 + z^p + z^{2p} + (1+z^p)x^p)\\
        &= x^{p-1}p(1+z^p)(1+z)) -2p(1 + z^p + z^{2p}) - x^pp(1+z^p) 
    \end{align*}
    has at most one zero. This follows from Descarte's rule of signs for pseudo-polynomials, since there is only one sign change in the coefficients.  \medskip

    \noindent \textbf{Proof of inequality \eqref{ineq: xneqy}:} We set again $z = xy$, so that we have to show
    $$
        1 + z^p + z^{2p} + y^p(1 + z^p)   \ge (1 + \frac{z}{y} + z)^p(1 + y + z)^p\,.
    $$
    The restriction $x \le y$ implies $0 \le z \le y^2$. At the endpoints the inequality holds, by Lemma \ref{lem: 1var}. As in the previous case, it is enough to check that the logarithmic derivative in $y$ of the ratio of left and right hand side vanishes at most once. It is 
    \begin{align*}
        \frac{p(1 + z^p)y^{p-1}}{1 + z^p + z^{2p} + y^p(1 + z^p)} + \frac{pzy^{-1}}{z + (1+z)y} - \frac{p}{1 + y + z}\,.
    \end{align*}
    Clearing all denominators, we obtain 
    \begin{align*}
        &\quad y^{-1} \cdot pz(1+z)(1 + z^p + z^{2p})\\
        &+ y^{p-1}\cdot (p(1+z^p)z(1+z) + p(1+z^p)z(1+z))\\
        &+ y^{p} \cdot (p(1+z^p)(1+z)^2 + p(1+z^p)z + pz(1+z^p) - pz(1+z^p))\\
        &+ y \cdot (-p(1+z)(1+z^p + z^{2p}))\,.
    \end{align*}
    This is pseudo-polynomial in $y$ with sign sequence of its coefficients $+++-$, so it has at most one zero. This completes the proof.
\end{proof}

Now we start the proof of Proposition \ref{prop: prek lein m2}.

\begin{proof}[Proof of Proposition \ref{prop: prek lein m2}]

The inequality \eqref{ineq: main m2} is homogeneous in $x$ and $y$, so we can assume that $x_0 + x_1 + x_2 = y_0 + y_1 + y_2 = 1$. We denote 
$$
    \mathbf{P} = \{x \in [0,1]^2 \, : \, \sum_{i = 0}^2 x_i = 1\}\,.
$$
\begin{lemma}
\label{lem: loc conc}
Suppose that $(x,y) \in \mathbf{P}^2$ minimizes 
$$
    J(x,y) := \sum_{k = 0}^4 \max_{i + j=k} x_i^p y_j^p
$$
on $\mathbf{P}^2$.
Then for each $i$, we have $x_i = 0$ or $x_i = 1$ or there exists $k$ and $i'$ such that $x_i y_{k - i} = x_{i'} y_{k - i'}$. Similarly, for each $j$ we have $y_j = 0$, $y_j = 1$ or there exists $k$ and $j'$ such that $x_{k-j}y_j = x_{k-j'}y_{j'}$.
\end{lemma}

\begin{proof}[Proof of Lemma \ref{lem: loc conc}]
Suppose that $(x,y)$ is a minimizer which does not have the claimed property. By symmetry in $x$ and $y$, we can then assume that there exists $i$ such that $x_i \in (0,1)$ and such that $x_i y_{k-i} \ne x_{i'}y_{k - i'}$ for all $k = i, i+1, i + 2$ and $i \ne i'$. Now we perturb $x$ to decrease $J(x,y)$ as follows: We define 
$$
    x_n(t) = 
    \begin{cases}
        (1 + t)x_n  &\text{if $n\ne i$}\\
        x_i - t(1-x_i) &\text{if $n = i$}
    \end{cases}\,.
$$
Since $x_i \in (0,1)$, there exists $\varepsilon > 0$ such that $x(t) \in \mathbf{P}$ for all $t \in (-\varepsilon, \varepsilon)$. Furthermore, for each $k$ with $x_i y_{k-i} < \max_{i' + j' = k} x_{i'}y_{j'}$, there exists another $\varepsilon > 0$ such that we have 
$$
    \max_{i'+ j' = k} x_{i'}(t)^py_{j'} = (1+t)^p \max_{i'+ j' = k} x_{i'}^py_{j'}
$$
for all $t \in (-\varepsilon, \varepsilon)$. Indeed, each product $x_{i'} y_{j'}$ with $i' \ne i$ changes by the same factor, and the product $x_i y_{k-i}$ is strictly smaller. Finally, if $x_i y_{k-i} = \max_{i' + j' = k} x_{i'}y_{j'}$, then there also exists $\varepsilon > 0$ such that we have 
$$
    x_i(t)^p y_{k-i}^p = \max_{i' + j' = k} x_{i'}(t)^py_{j'}^p
$$
for all $t \in (-\varepsilon, \varepsilon)$, since $x_i y_{k-i}$ is stricly larger than all other products $x_{i'} y_{k-i'}$. 

Taking the minimum of the $\varepsilon$'s, we obtain a neighbourhood $(-\varepsilon, \varepsilon)$ of $0$ on which 
$$
    J(x(t), y) = \sum_{k = 0}^4 \max_{i' + j' = k} x_{i'}(t)^p y_{j'}^p
$$
is a concave function of $t$. Thus $(x,y) = (x(0), y)$ cannot be a minimizer.
\end{proof}

We now pick a minimizer $(x,y)$. We first assume that $x_i, y_j \in (0,1)$ for all $i, j$. By Lemma \ref{lem: loc conc}, at least two equations $x_i y_{k-i} = x_{i'} y_{k - i'}$ must hold, involving each of $x_0, x_1$ and $x_2$. It is easy to check that all such pairs of equations are
\begin{enumerate}
    \item[1)] $x_0y_1 = x_1y_0$ and $x_0y_2 = x_2 y_0$,
    \item[2)] $x_0y_1 = x_1y_0$ and $x_1y_1 = x_2y_0$,
    \item[3)] $x_0y_1 = x_1y_0$ and $x_1y_2 = x_2y_1$,
    \item[4)] $x_0y_2 = x_1y_1 = x_2y_0$,
    \item[5)] $x_0y_2 = x_1y_1$ and $x_1y_2 = x_2y_1$,
    \item[6)] $x_0y_2 = x_2y_0$ and $x_1y_2 = x_2y_1$.
\end{enumerate}
Pairs 1), 3) and 6) imply that $x = y$.
In pair 2) there is no equation involving $y_2$. Adding any of the three possible equations one obtains that $x = y$. Pair 5) is symmetric to pair 2) after reversing the order of the variables, so it also reduces to $x = y$. Writing $x = y = (x_0, x_0a, x_0ab)$ or $x = y = (x_0, x_0b, x_0ba)$ with $0 \le a \le b$, the inequality $J(x,y) \ge 1$ simplifies to \eqref{ineq: xeqy1} and \eqref{ineq: xeqy2}, respectively, so it holds by Lemma \ref{lem: 2var}.

The only remaining case is pair 4). In that case we can write $x = (x_0, ax_0, abx_0)$ and $y = (y_0, b y_0, aby_0)$, and by symmetry we have without loss of generality $0 \le a \le b$. Then the inequality simplifies to \eqref{ineq: xneqy}, so follows again from Lemma \ref{lem: 2var}.
\medskip

Now we consider minimizers $(x,y)$ on the boundary of $\mathbf{P}^2$. If $(x,y)$ is on the boundary, then at least one of the variables is $0$. Since $x, y$ are nondecreasing, this must be $x_0$ or $y_0$. We assume without loss of generality that $y_0 = 0$.

If $y_1 = 0$ or then $y = (0,0,1)$, and the inequality holds. So we can assume that $y_1, y_2 \in (0,1)$. If any of $x_0, x_1, x_2$ equals $1$, the inequality is trivial. If $x_0 = 0$, then the inequality reduces to the case $m = 1$, which is proven in (\ref{npoint}). 
If $x_0 \ne 0$ then $x_0, x_1, x_2 \in (0,1)$ since $x$ is nondecreasing. So by Lemma \ref{lem: loc conc}, one of the $6$ pairs of equations listed above must hold. Pairs 1), 2), 3), 4) and 6) contradict the assumption that $y_0 = 0$ and all other variables are not $0$. If pair 5) holds, then we can write 
$$
    x = (x_0, ax_0, a^2x_0)\,, \quad \quad y = (0, y_1, ay_1)\,.
$$
Then the inequality reduces to \eqref{ineq: m2b} above, so it follows from Lemma \ref{lem: 1var}.
\end{proof}

\subsection{The proof of Theorem~\ref{trisets}}
By Remark~\ref{genrem1} and Lemma 11.4 in \cite{MRSZ} (which is similar to Lemma~\ref{perestanovka}) it suffices to show that for all non-negative finitely supported $g, h :\mathbb{Z} \to [0, \infty)$ and $f(x) = 1_{\{0\}}+\delta 1_{\{1\}}$ with any $\delta \in [0,1]$, we have 
\begin{align*}
    \frac{\|f\bar{*}g\bar{*}h\|_{1}}{\|g\|_{p}\|h\|_{q}}\geq (1+\delta^{c})^{1/c}.
\end{align*}
Theorem 11.1 in \cite{MRSZ} implies that 
\begin{align*}
    \frac{\|f\bar{*}g\bar{*}h\|_{1}}{\|g\|_{p}\|h\|_{q}} \geq c_{\delta}:=\frac{(1-\delta^{p})^{1/p}(1-\delta^{q})^{1/q}}{1-\delta}. 
\end{align*}
The case $\delta=1$ should be understood as $\lim_{\delta \to 1}c_{\delta}=p^{1/p}q^{1/q}$. Thus it suffices prove the following subtle inequality. 

\begin{lemma}
    For any $p, q > 1$ satisfying $1/p+1/q=1$ one has 
    \[
    \frac{(1-x^p)^{1/p}(1-x^q)^{1/q}}{(1-x)(1+x^c)^{1/c}}\geq 1,
    \]
    for any $x\in [0, 1]$, where $c:=1/\log_2{(p^{1/p}q^{1/q})}$ is chosen in such a way that $x\rightarrow 1$ gives an equality. 
\end{lemma}
\begin{proof}
    Denote the left hand side of the inequality by $F(x)$. We want to prove that $\log{F(x)}\geq 0$ where $\log{F(x)}$ is given by
    \[
    \log{F(x)}=\tfrac{1}{p}\log{|1-x^p|} + \tfrac{1}{q}\log{|1-x^q|} - \log{|1-x|} - \tfrac{1}{c}\log{|1+x^c|}.
    \]
    We have 
    \begin{align*}
        x\cdot \partial_x \log{F(x)} &=
        -\frac{x^p}{1-x^p}-\frac{x^q}{1-x^q} + \frac{x}{1-x} - \frac{x^c}{1+x^c}
        \\&=
        \frac{1}{1-x}+\frac{1}{1+x^c}-\frac{1}{1-x^p}-\frac{1}{1-x^q},
    \end{align*}
    and so $H(x):=x(1-x)(1+x^c)(1-x^p)(1-x^q)\cdot \partial_x \log{F(x)}$ can be written as a pseudo-polynomial
    \[
    x-x^c+2x^{1+c}-x^p-x^q-x^{1+c+p}-x^{1+c+q}+2x^{p+q}-x^{p+q+1}+x^{p+q+c}.
    \]
    We have $H(1/x)=H(x)\cdot x^{-1-c-p-q}$ and by the rule of signs $H(x)$ has at most 6 positive roots. Indeed, it's easy to see that $c\geq 1$ by Jensen's inequality, and so the smallest exponent is $1$ and the largest is $x^{p+q+c}$. When arranged in the increasing order of exponents, they have at most one neighbouring term with a negative sign each, and the remaining two terms with a plus sign (i.e. $2x^{1+c}$ and $2x^{p+q}$) give at most 4 more changes of sign.

    Note that as $x\rightarrow 0$ we have
    \[
    F(x)=\frac{(1-x^p/p)(1-x^q/q)}{(1-x)}\cdot (1+o_x(1)) = \frac{1-o(x)}{1-x},
    \]
    and so $F(x)\geq 1$ in some neighbourhood of zero. 

    Now assume that $F(t) < 1$ for some $t\in (0,1)$. Then $H(t)$ must have a root on both $(0, t)$ and $(t, 1)$. Hence, together with a $x\mapsto 1/x$ symmetry and the triple root at zero we have at least 7 positive roots leading to a contradiction. 
\end{proof}

\subsection{The proof of Theorem~\ref{main thm 2}}
 The following lemma plays an important role in the proof of Theorem~\ref{main thm 2}, it is motivated by Theorem 1.1 in \cite{GMRSZ} and can be established with a similar strategy. 
 
\begin{lemma}\label{key ingridient in main thm 2}
Let $d\geq 1$, $k\in\N$, $A_1, A_2,\dots,A_{k} \subseteq \Z^d$ be finite sets, and suppose that $U=\{0,1,\dots,k-1\}^d$. Then $|A_1+A_2+\dots+A_k+U| \geq |U|\prod_{i=1}^{k}|A_i|^{1/k}$.
\end{lemma}

Notice that the inequality in Lemma \eqref{key ingridient in main thm 2} is optimal, in fact if $A_{1}=A_2=\dots=A_k=U=\{0,1,\dots,k-1\}^d$, then, both sides are equal to $k^{2d
}$. Before we verify the lemma let us complete the proof of Theorem~\ref{main thm 2}.

The upper bound in Theorem~\ref{main thm 2} can be obtained by considering the case $A=\{0,1,\dots,k-1\}^d$. We will focus our attention on the lower bound.
Let $\alpha:=\frac{|kA|}{|A|}$. By the Pl\"unnecke inequality
for different summands in commutative group \cite[Theorem 1.1]{GMR} there is a non-empty subset
$B$ of $A$ such that
$$
\alpha^{\frac{k}{k-1}}|B|\geq
|B+kA|.
$$
Then, by Lemma \ref{key ingridient in main thm 2} we have
\begin{align*}
\alpha^{\frac{k}{k-1}}|B|&\geq
|B+kA|\\
&\geq|B+(k-1)A+\{0,1,\dots,k\}^d|\\
&\geq k^d|B|^{\frac{1}{k}}|A|^{\frac{k-1}{k}}.
\end{align*}
Therefore
$$
\alpha^{\frac{k}{k-1}}|B|^{\frac{k-1}{k}}\geq k^d|A|^{\frac{k-1}{k}}\geq k^d|B|^{\frac{k-1}{k}}.
$$
Thus, we conclude that
$$
\alpha\geq k^{\frac{(k-1)d}{k}}.
$$




\subsubsection{Proof of Lemma \ref{key ingridient in main thm 2}}

The proof of this lemma follows from an adaptation of the proof of \cite[Theorem 1.2]{GMRSZ} by Green, Matolcsi, Rusza, Shakan and Zhelezov. Similarly to \cite{GMRSZ}, we obtain Lemma \ref{key ingridient in main thm 2} by discretizing the continuous  Pr\'ekopa--Leindler inequality for many functions. 


\begin{proposition}\label{prop01}
Let $k\in\N$, $a_1, a_2,\dots, a_{k}  : \Z \rightarrow [0,\infty)$ be compactly supported functions, let $\lambda>0$ and $p_j:=e^{j\lambda}\frac{e^{\lambda}-1}{e^{k\lambda}-1}$ for all $j\in\{0,1,\dots,k-1\}$. Then we have
\begin{equation}\label{key ineq} 
\sum_{n\in\Z} \max_{j\in\{0,1,\dots,k-1\}}(p_{j} a_1\overast a_2\overast\dots\overast a_{k}(n-j)) \geq \prod_{i=1}^{k}\Vert a_i \Vert_k.
\end{equation}
\end{proposition}

\begin{corollary}
In particular, taking the limit in the left hand side of \eqref{key ineq} as $\lambda\to 0^{+}$,  we obtain that \eqref{key ineq} holds when $p_0=p_1=\dots=p_{k-1}=\frac{1}{k}$.
\end{corollary}

Assume that the functions $f_i : \R \rightarrow [0, \infty)$ are compactly supported and piece-wise continuous for each $i\in\{1,2,\dots,k\}$. Then the (1-dimensional) generalized Pr\'ekopa--Leindler inequality states that
\begin{equation}\label{pl}  \int f_1 \overast f_2 \overast\dots\overast f_k \geq k \prod_{i=1}^{k}\Vert f_i \Vert_k , \end{equation} where the left hand side denotes the max-convolution, which is defined by
\[ f_1 \overast f_2\overast\dots\overast f_k(x) := \sup_{x_1+x_2+\dots+x_k; x_i \in \R} \prod_{i=1}^{k} f_i(x_i),\] and as usual
\[ \Vert f \Vert_k := \left( \int_{\R} |f(x)|^k dx\right)^{1/k}.\]

\begin{proof}[Proof of Proposition \ref{prop01}]  
We use \eqref{pl} with the functions $f_{i}$ defined by
\[ f_{i}(x) := e^{\lambda \{x \}} a_{i} (\lfloor x\rfloor),\]
for all $i\in\{1,2,\dots,k\}$.

Let $n \in \Z$ and $0 \leq t < 1$. Assume that $\sum_{i=1}^{k}x_i= n + t$. Then, since $x - 1 < \lfloor x\rfloor \leq x$, we have $n-k < \sum_{i=1}^{k}\lfloor x_{i}\rfloor < n+1$. 
If $\sum_{i=1}^{k}\lfloor x_i \rfloor=n-m$ for some $m\in\{0,1,\dots,k-1\}$, then 
\[ \prod_{i=1}^{k} f_i(x_i) \leq e^{\lambda(t + m)} a_1 \overast a_2 \overast\dots \overast a_{k} (n-m).\]
Thus
\[ f_1 \overast f_2 \overast\dots \overast f_k (n + t) \leq e^{\lambda t} \max_{m\in\{0,1,\dots,k-1\}} e^{\lambda m} a_1 \overline{\ast} a_2 \overast \dots\overast a_{k} (n-m).\]
Taking integrals from 0 to 1 with respect to $t$ and then adding over all $n \in \Z$ we obtain
\begin{equation}\label{star-1} 
\int_{\R} f_1 \overast f_2 \overast\dots \overast f_k(x) dx\leq \frac{e^{\lambda} - 1}{\lambda} \sum_{n\in \mathbb{Z}} \max_{m\in\{0,1,\dots,k-1\}} e^{\lambda m} a_1 \overline{\ast} a_2 \overast \dots\overast a_{k} (n-m).\end{equation}
On the other hand,
\[ \Vert f_i \Vert_k^k = \frac{e^{k\lambda} - 1}{k\lambda} \Vert a_i \Vert_k^k,\]
for each $i\in\{1,2,\dots,k\}$. 
Therefore, from \eqref{pl} we obtain
\[ 
\sum_{n\in\mathbb{Z}} \max_{m\in\{0,1,\dots,k-1\}} e^{\lambda m} a_1 \overline{\ast} a_2 \overast \dots\overast a_{k} (n-m) \geq \frac{e^{k\lambda} - 1}{e^{\lambda}-1}\prod_{i=1}^{k} \Vert a_i \Vert_k.\] The result then follows from the fact that $p_j:=e^{j\lambda}\frac{e^{\lambda}-1}{e^{k\lambda}-1}$. 
\end{proof}


\begin{proof}[Proof of Lemma \ref{key ingridient in main thm 2}]

We proceed by induction on $d$. The proof of the inductive step also proves the base case $d = 1$.


Recall that $U=\{0,1,\dots,k-1\}^d$. We define $\pi:U\to\{0,1,\dots,k-1\}$ to be the projection onto the last coordinate, i.e
$\pi(x_1,\dots,x_d)=x_d$ for all $(x_1,\dots,x_d)\in\{0,1,\dots,k-1\}^d$. For each $i\in\{1,2,\dots,k\}$ and $x\in\{0,1,\dots,k-1\}$, let $A_{i,x} := A_i \cap \pi^{-1}(x)$ be the fibre of $A_i$ above $x$. The set $U$ has $k$ fibres $U_0, U_1, \dots, U_{k-1}$ and 
they are all $d-1$-dimension cubes, here $U_i=\{0,1,\dots,k-1\}^{d-1}\times\{i\}$ for each $i\in\{0,1,\dots,k-1\}$. Moreover, we have that $\frac{|U_i|}{|U|}=\frac{1}{k}$ for each $i\in\{0,1,\dots,k-1\}$.

We note that the fibre of $A_1+A_2+\dots+A_k+U$ above $n$ contains the set $A_{1,x_1}+A_{2,x_2}+\dots+A_{k,x_k}+U_m$ whenever $x_1+x_2+\dots+x_k=n-m$ for $m\in\{0,1,\dots,k-1\}$. By induction,
\[ |A_{1,x_1}+A_{2,x_2}+\dots+A_{k,x_k}+U_m| \geq |U_m|\prod_{i=1}^{k}|A_{i,x_{i}}|^{1/k},\]
for all $m\in\{0,1,\dots,k-1\}$.\\

For each $i\in\{1,2,\dots,k\}$ we define $a_{i}:\Z\to[0,+\infty)$ by $a_{i}(x):=|A_{i,x}|^{1/k}$ for all $x\in\{0,1,\dots,k-1\}$ and $a_i(x)=0$ otherwise. Then summing over $n$, using the fact that $|U_m|=\frac{|U|}{k}$ for all $0\leq m\leq k-1$, and using Proposition \ref{prop01}, we obtain
\begin{align*}
& |A_{1}+A_2+\dots+A_{k}+U|\\
& = \sum_{n\in\Z} |(A_{1}+A_2+\dots+A_{k}+U)_n| \\ 
&\geq \sum_{n} 
\frac{|U|}{k}
\max_{x_1+x_2+\dots+x_{k}=n-m}\prod_{i=1}^{k}|A_{i,x_i}|^{1/k}
\\
& \geq |U| \sum_n \max_{m\in\{0,1,\dots,k-1\}} \big( \frac{1}{k} a_{1}\overast a_{2}\overast\dots \overast a_{k}(n-m)\big) \\ 
& \geq |U| \prod_{i=1}^{k}\Vert a_i \Vert_k  = |U|\prod_{i=1}^{k}|A_{i}|^{1/k}. 
\end{align*}

\end{proof}

\section{Acknowledgments}
The authors are very thankful to Ben Green, Jaume de Dios Pont, Pavlos Kalantzopoulos, Fedor Petrov, and Terence Tao for helpful discussions. L.B. was funded by the Deutsche Forschungsgemeinschaft (DFG, German Research Foundation) under Germany's Excellence Strategy -- EXC-2047/1 -- 390685813 as well as SFB 1060. P.I. acknowledges support from NSF grant CAREER-DMS-2152401.


\begin{thebibliography}{99}



\bibitem{BDFKK} J. Bourgain, S. J. Dilworth, K. Ford, S. Konyagin, and D. Kutzarova,   
\emph{Explicit constructions of RIP matrices and related problems,} Duke Math. Journal 159(1): 145--185 (2011).


\bibitem{BRO} G.~Brown, M.~S.~Keane, W.~Moran, C.~E.~M.~Pierce  
\emph{An inequality, with applications to Cantor measures and normal numbers,} Mathematika 35 (1988), no. 1, 87--94.







\bibitem{DGIM} J. de Dios, R. Greenfeld, P. Ivasnisvili and J. Madrid, \emph{Additive energies on discrete cubes,} Preprint to appear in Discrete Analysis.

\bibitem{FLS} S. Fish, Ben Lund, and A. Sheffer,   
\emph{A Construction for Difference Sets with Local Properties,} European Journal of Combinatorics, 79 (2019), 237--243


\bibitem{GMR} K. Gyarmati, M.~Matolcsi and I.~Z.~Ruzsa, \emph{Pl\"unnecke’s Inequality for Different Summands,} Bolyai Society Mathematical Studies book series (BSMS,volume 19), Building Bridges,
Between Mathematics and Computer Science, pages 309--320.

\bibitem{BG2} B. Green, \emph{Waring's problem with restricted digits,} preprint (2023). arXiv: 309.09383



\bibitem{GMRSZ} B. Green, D.~Matolcsi, I.~Z.~Ruzsa, G.~Shakan and D.~Zhelezov, \emph{A  weighted Prekopa-Leindler inequality and sumsets with quasicubes,} preprint.

\bibitem{GT} B. J.~Green and T. C. Tao, \emph{Compressions, convex geometry and the Freiman-Bilu theorem,} Q. J. Math. 57 (2006), no. 4, 495--504.


\bibitem{HS1}  D.~Hajela,  P.~Seymour, \emph{Counting points in hypercubes and convolution measure algebras,} Combinatorica 5, no. 3, 205--214.

\bibitem{HP} B. Hanson and G. Petridis, \emph{A Question of Bukh on Sums of Dilates,} Discrete Analysis, 2021: Paper No. 13, 21 pp.

\bibitem{I} P.  Ivanisvili,   
\emph{Convolution estimates and the number of disjoint
partitions,} The  Electronic  Journal  of Combinatorics, Volume 24, Issue 2 (2017), Paper P2.43.

\bibitem{IV} P.  Ivanisvili, A.~Volberg,    
\emph{Bellman partial differential equation and the hill property for classical isoperimetric problems,} arXiv: 1506.03409

\bibitem{KT} D.  Kane and  T.  Tao,   
\emph{A  bound  on  Partitioning  Clusters,} The  Electronic  Journal  of Combinatorics, Volume 24, Issue 2 (2017), Paper P2.31.

\bibitem{Kath} R. Kathory,   
\emph{Estimate of Minkowski sum} 
https://mathoverflow.net/q/453574



\bibitem{K} V. Kovac,   
\emph{On binomial sums, additive energies, and lazy random walks,} 
Preprint at arxiv.org/abs/2206.01591


\bibitem{LLS1} H.~J.~Landau, B.~F.~Logan and L.~A.~Shepp,   
\emph{An inequality conjectured by Hajela and Seymour arising in combinatorial geometry,} 
Combinatorica 5 (1985), no. 4, 337--342.





\bibitem{MRSZ} D.~Matolcsi, I.~Z.~Ruzsa, G.~Shakan and D.~Zhelezov, \emph{An analytic approach to cardinalities of sumsets,} Combinatorica (2022). https://doi.org/10.1007/s00493-021-4547-0.

\bibitem{Polya}
G.~P\'olya, G.~Szeg\"o, \emph{Problems and Theorems in Analysis}, vol. II, Springer-Verlag, New York, 1976

\bibitem{Shakan1}
G.~Shakan, \emph{Nice proof of inequality $(1-x^p)^{1/p}(1-x^q)^{1/q}\ge (1-x)(1+x^c)^{1/c}$ where $2^{1/c} = p^{1/p} q^{1/q}$?}, MathOverflow, 2019. URL:https://mathoverflow.net/q/343334 (version: 2019-10-08)

\bibitem{TaoVu}
T. Tao, V. Vu, Additive combinatorics. Cambridge Studies in Advanced Mathematics, 105. Cambridge University Press, Cambridge, 2006.


\bibitem{W} 
D. R. Woodall, 
\emph{A theorem on cubes,} Mathematika 24 (1977), 60–-62.

\bibitem{DZDP} 
D.~Zhelezov, D\"om\"ot\"or P\'alv\"olgyi
\emph{Query complexity and the polynomial Freiman–Ruzsa conjecture,} Advances in Mathematics, Volume 392, 3 December 2021. 



\end{thebibliography}
\end{document}